\documentclass[journal, twoside]{IEEEtran}

\ifCLASSINFOpdf
\else
\fi

\usepackage{vector, array, amssymb, amscd, amsthm}
\usepackage[cmex10]{amsmath}
\usepackage{algorithmic}
\usepackage{url}
\usepackage{cite} 
\usepackage[geometry]{ifsym}
\usepackage{bbm}
\usepackage{algorithm}
\usepackage{algorithmic}
\usepackage{enumerate}
\usepackage{color}
\usepackage{tikz}
\usepackage{booktabs}
\usetikzlibrary{calc,decorations.pathmorphing}
\usepackage{subfig}
\usepackage{graphicx}

\let\oldFootnote\footnote
\newcommand\nextToken\relax

\renewcommand\footnote[1]{%
    \oldFootnote{#1}\futurelet\nextToken\isFootnote}

\newcommand\isFootnote{%
    \ifx\footnote\nextToken\textsuperscript{,}\fi}

\newtheorem{assumption}{Assumption}
\newtheorem{proposition}{Proposition}
\newtheorem{lemma}{Lemma}
\newtheorem{thm}{Theorem}
\newtheorem{cor}{Corollary}

\newtheorem{remark}{Remark}



\newcommand{\zbf}{z}

\newcommand{\csbf}{\xi}
\newcommand{\dbf}{d}

\renewcommand{\baselinestretch}{1}

\begin{document}

\title{Regularized Jacobi iteration for decentralized convex optimization with separable constraints}

\author{Luca~Deori, Kostas~Margellos and Maria~Prandini
\thanks{Research was supported by the European Commission, H2020, under the project UnCoVerCPS, grant number 643921.}\vspace{-1\baselineskip}
\thanks{L. Deori and M. Prandini are with the Dipartimento di Elettronica Informazione e Bioingegneria, Politecnico di Milano,
Piazza Leonardo da Vinci 32, 20133 Milano, Italy, e-mail: \texttt{\{luca.deori, maria.prandini\}@polimi.it}

K. Margellos is with the Department of Engineering Science, University of Oxford, Parks Road,
Oxford, OX1 3PJ, United Kingdom, e-mail: \texttt{kostas.margellos@eng.ox.ac.uk}
}}

\maketitle
\IEEEpeerreviewmaketitle

\begin{abstract}
We consider multi-agent, convex optimization programs subject to separable constraints, where the constraint function of each agent involves only its local decision vector, while the decision vectors of all agents are coupled via a common objective function. We focus on a regularized variant of the so called Jacobi algorithm for decentralized computation in such problems. We first consider the case where the objective function is quadratic, and provide a fixed-point theoretic analysis showing that the algorithm converges to a minimizer of the centralized problem. Moreover, we quantify the potential benefits of such an iterative scheme by comparing it against a scaled projected gradient algorithm.
We then consider the general case and show that all limit points of the proposed iteration are optimal solutions of the centralized problem.
The efficacy of the proposed algorithm is illustrated by applying it to the problem of optimal charging of electric vehicles, where, as opposed to earlier approaches, we show convergence to an optimal charging scheme for a finite, possibly large, number of vehicles.
\end{abstract}


\section{Introduction} \label{sec:secI}
\IEEEPARstart{O}{ptimization} in multi-agent systems has attracted significant attention in the control and operations research communities, due to its applicability to different domains, e.g., energy systems \cite{Bagagiolo_Bauso_2014}, \cite{Gharesifard_etal_2016}, mobility systems \cite{Hiskens_2013}, \cite{Gan_etal_2013}, \cite{Franc_2014}, robotic networks \cite{Pavone_etal_2007}, etc. In this paper we focus on a specific class of multi-agent optimization programs that are convex and are subject to constraints that are separable, i.e., the constraint function of each agent involves only its local decision vector. The agents' decision vectors are, however, coupled by means of a common objective function.
The considered structure, although specific, captures a wide class of engineering problems, like the electric vehicle optimal charging problem studied in this paper. Solving such problems in a centralized fashion would require agents to share their local constraint functions with each other, while even if this was possible it would unnecessarily increase the computational burden.


To allow for a computationally tractable solution, while accounting for information sharing issues, we adopt an iterative, decentralized perspective, where agents perform local computations in parallel, and then exchange with each other their new solutions, or broadcast them to some central authority that sends an update to each agent.
Admittedly, distributed optimization offers a more general communication setup, however, the fact that agents decision vectors are coupled via the objective function poses additional difficulties, preventing the use of standard distributed algorithms \cite{Nedic_etal_2010}, \cite{Zhu_Martinez_2012}.
Even upon an epigraphic reformulation, the resulting problem will not exhibit the structure typically encountered in distributed optimization since the resulting coupling constraint will not necessarily be of ``budget'' form as required, e.g., in \cite{Koshal_2011}.

\subsection{Related work}
From a cooperative optimization point of view, algorithms for decentralized solutions to convex optimization problems with separable constraints can be found in \cite{Bertsekas_Tsitsiklis_1997,Parikh_Boyd_2013}, and references therein. Two main algorithmic directions can be distinguished, both of them relying on an iterative process. The first one is based on each agent performing at every iteration a local gradient descent step, while keeping the decision variables of all other agents fixed to the values communicated at the previous iteration \cite{Goldstein_1964,Levitin_Poljak_1965,Bertsekas_1976a}. Under certain structural assumptions (differentiability of the objective function and Lipschitz continuity of its gradient), it is shown that this scheme converges to some minimizer of the centralized problem, for an appropriately chosen gradient step-size.


The second direction for decentralized optimization involves mainly the so called Jacobi algorithm, which serves as an alternative to gradient algorithms. The Gauss-Seidel algorithm exhibits similarities with the Jacobi one, but is not of parallelizable nature \cite{Attouch_etal_2013}, unless a coloring scheme is adopted (see Section 1.2.4 in \cite{Bertsekas_Tsitsiklis_1997}). Under the Jacobi algorithmic setup, at every iteration, instead of performing a gradient step, each agent minimizes the common objective function subject to its local constraints, while keeping the decision vectors of all other agents fixed to their values at the previous iteration. A regularized version of the Jacobi algorithm has been proposed in \cite{Cohen1978,Cohen1980}, and more recently in \cite{Patriksson1998,Zhu1995}. Other parallelizable iterative methods are proposed in \cite{Gan_etal_2013,Tarak2016,Necoara2016}, where, however, partially separable cost functions are considered.


From a non-cooperative perspective there has recently been a notable research activity using tools from mean-field and aggregative game theory. Under a deterministic, discrete-time setting like the one considered in the present paper, \cite{Hiskens_2013,Franc_2014,Grammatico_etal_2015} deal with the non-cooperative counterpart of our work, for the case of quadratic objective functions.
In all cases, the considered algorithm is shown to converge not to a minimizer, but to an approximate Nash equilibrium of a related game, and to an exact Nash equilibrium in the limiting case where the number of agents tends to infinity. The recent work of \cite{Paccagnan_etal_2016} shows convergence to an exact Nash equilibrium for a finite number of agents.

\subsection{Contributions of this work and organization of the paper}
In this paper we adopt a cooperative point of view, and consider a regularized Jacobi algorithm similar to the one in \cite{Cohen1978,Cohen1980}.
Our contributions can be summarized as follows:
\begin{enumerate}
\item We establish an equivalence between the set of minimizers of the problem under study and the set of fixed-points of the mapping induced by the considered regularized Jacobi algorithm.
\item For the case where the objective function is quadratic we follow a fixed-point theoretic analysis and show convergence to an optimal solution of the centralized problem counterpart. As opposed to \cite{Cohen1978,Cohen1980}, we provide an explicit calculation of the regularization coefficient that ensures convergence, and show that the convergence properties of this approach outperform the ones of scaled projected gradient algorithms. As such, our algorithm not only serves as the cooperative counterpart of \cite{Paccagnan_etal_2016}, but also enjoys superior convergence properties.
\item In the case of general differentiable objective functions, we provide a proof that all limit points of the proposed iteration are optimal solutions of the centralized problem counterpart; such a proof is not provided in \cite{Cohen1978,Cohen1980}.
\item We extend the results of \cite{Hiskens_2013,Franc_2014,Grammatico_etal_2015} on electric vehicle charging control, achieving convergence to an optimal charging scheme with a finite number of vehicles. This serves also as an extension of \cite{Gan_etal_2013}, where convergence to the optimal objective value and not to the optimal charging solution was provided.
\end{enumerate}
The results obtained here extend significantly our earlier work in \cite{Deori_etal_2016}, where only the case of quadratic functions was considered, omitting various proofs due to space limitations,  and no formal comparison with the gradient methods was provided.

The rest of the paper is organized as follows.
Section \ref{sec:secII} introduces the problem under study and states the proposed algorithm. In Section \ref{sec:secIII} we provide the main convergence result for the case where the objective function is quadratic, and a comparison with scaled projected gradient methods. Section \ref{sec:secIV} provides a convergence analysis for the general case of differentiable objective function, while the proof is deferred to the Appendix. Section \ref{sec:secV} provides application of the developed scheme to the problem of optimal charging of electric vehicles and includes an extensive simulation study, while Section \ref{sec:secVI} concludes the paper and outlines some directions for future research.

\section{Decentralized problem formulation} \label{sec:secII}
\subsection{Problem statement} \label{sec:statement}
We consider the following multi-agent constrained optimization problem
\begin{align}
\mathcal{P}: &\min_{\{x^i \in \mathbb{R}^{n_i}\}_{i=1}^m} f(x^1,\ldots,x^m) \label{eq:Pobj} \\
&\text{subject to } \nonumber \\
&x^i \in X^i, \text{ for all } i=1,\ldots,m, \label{eq:Pcon}
\end{align}
where each agent $i$, $i=1,2,\dots,m$, has a local decision vector $x^i\in \mathbb{R}^{n_i}$ and a local constraint set $X^i\subseteq \mathbb{R}^{n_i}$, and cooperates to determine a minimizer of
$f: \mathbb{R}^{n_1} \times \ldots \times \mathbb{R}^{n_m} \to \mathbb{R}$, which couples its decision vector with those of the other agents.

We study, in particular, the case when the following assumption holds.
\begin{assumption}\label{ass:Convexity}
The objective function $f: \mathbb{R}^{n_1} \times \ldots \times \mathbb{R}^{n_m} \to \mathbb{R}$ is given by
$$
f(x^1,\ldots,x^m)= x^\top Q x + q^\top x ,
$$
where $x=[x^{1,\top},\ldots,x^{m,\top}]^\top \in \mathbb{R}^{n}$ with $n = \sum_{i=1}^m n_i$,   $Q \in \mathbb{R}^{n\times n}$ is symmetric and positive definite ($Q=Q^\top\succeq 0$) and $q\in\mathbb{R}^n$. Moreover, the sets $X^i \subseteq \mathbb{R}^{n_i}$, $i=1,\ldots,m$, are non-empty, compact and convex.
\end{assumption}

Note that $Q$ is assumed to be symmetric without loss of generality; in the opposite case it could be split in a symmetric and an antisymmetric part, with the latter giving rise to terms that simplify each other.

\begin{remark}[Problem generalization]

 The considered framework allows for objective functions of the form $f(x^1,\ldots,x^m)= x^\top Q x + q^\top x + \sum_{i=0}^m g^i(x^i)$, where the $g^i(x^i)$ are convex functions that depend on the local decision vectors $x^i$ only, $i = 1,\ldots, m$, and may be useful to encode a utility function for each agent. In this case an epigraphic reformulation can be exploited to bring the cost back to be quadratic, while preserving constraint separability. More precisely, by introducing an additional local variable, say $h^i$, in the decision vector $y^i=[x^{i,\top} \ h^i]^\top$ of agent $i$, the local constraint set can be defined as $Y^i=X^i \cap \{g^i(x^i) \le h^i\}$, while the objective function can be rewritten as $ x^\top Q x + q^\top x + \sum_{i=0}^m h^i$, which is quadratic in $y=[y^{1,\top} \ \dots \ y^{m,\top}]^\top$.
\end{remark}

Under Assumption \ref{ass:Convexity}, given that function $f$ is continuously differentiable and convex and the constraint set $X=X^1\times \dots \times X^m$ is non-empty and compact, by the Weierstrass' theorem (\cite[Proposition A8, p. $625$]{Bertsekas_Tsitsiklis_1997}), $\mathcal{P}$ admits at least one optimal solution. However, $\mathcal{P}$ does not necessarily admit a unique minimizer.

With a slight abuse of notation, for each $i$, $i=1,\ldots,m$, let $f(\cdot,x^{-i}): \mathbb{R}^{n_i} \to \mathbb{R}$ be the objective function in \eqref{eq:Pobj} as a function of the decision vector $x^i$ of agent $i$, when the decision vectors of all other agents are fixed to $x^{-i} \in \mathbb{R}^{n-n_i}$. We will occasionally also write $f(x)$ instead of $f(x^1,\ldots,x^m)$. We will use these notations interchangeably, but the interpretation will always be clear from the context.


\subsection{Regularized Jacobi algorithm} \label{sec:secIIB}

\begin{algorithm}[t]
\caption{Decentralized algorithm}
\begin{algorithmic}[1]
\STATE \textbf{Initialization} \\
\STATE ~~$k=0$. \\
\STATE ~~Consider $x^i_0 \in X^i$, for all $i=1,\ldots,m$. \\
\STATE \textbf{For $i=1,\ldots,m$ repeat until convergence} \\
\STATE ~~Agent $i$ receives $x^{-i}_k$ from central authority. \\
\STATE ~~$x^i_{k+1} = \arg \min_{z^i \in X^i}\!\! \left\{f(z^i,x^{-i}_k) + c \|z^i - x^i_k\|^2\right\} $. \\
\STATE ~~$k \leftarrow k+1$.
\end{algorithmic}
\label{alg:Alg1}
\end{algorithm}

Solving problem $\mathcal{P}$  in a centralized fashion is not always possible since agents may not be willing to share $X^i$, $i=1,\ldots,m$. Moreover, even if this was the case, solving $\mathcal{P}$ in one shot might be computationally challenging. To overcome this and account for information sharing issues, motivated by the particular structure of $\mathcal{P}$ with separable constraint sets, we follow a decentralized, iterative approach as described in Algorithm \ref{alg:Alg1}.

Initially, each agent $i$, $i=1,\ldots,m$, starts with some value $x^i_0 \in X^i$, such that $\big ( x^1_0,\ldots,x^m_0\big )$ is feasible and constitutes an estimate of what the minimizer of $\mathcal{P}$ might be (step 3, Algorithm \ref{alg:Alg1}). At iteration $k+1$, each agent $i$ receives the values of all other agents $x^{-i}_k$ (step 5, Algorithm \ref{alg:Alg1}) from the central authority, and updates the estimate of its own decision vector $x^i$ by solving a local minimization problem (step 6, Algorithm \ref{alg:Alg1}). The performance criterion in this local problem is a linear combination of the objective $f(z^i,x^{-i}_k)$, where the variables of all other agents apart from the $i$-th one are fixed to their values at iteration
$k$, and a quadratic regularization term, penalizing the difference between the decision vector $z^i$ and the value of agent's $i$ own variable at iteration $k$, i.e., $x^i_k$. The relative importance of these two terms is dictated by the regularization coefficient $c \in \mathbb{R}_+$, which plays a key role in determining the convergence properties of  Algorithm \ref{alg:Alg1}.
Note that under Assumption \ref{ass:Convexity}, and due to the presence of the quadratic penalty term, the resulting problem is strictly convex with respect to $z^i$, and hence admits a unique minimizer.

\begin{remark}[Information exchange]
To implement Algorithm \ref{alg:Alg1}, at iteration $k+1$, it is needed that some central authority, or common processing node, collects and broadcasts the current solution of each agent to all others, and that the agents have knowledge of the common objective function $f$ so that each of them can compute $f(\cdot,x^{- i}_k)$ (alternatively the central authority can broadcast it to each agent $i$, $i=1,\ldots,m$).
However, in the case where objective functions that are coupled only through the average of some variables, the central authority needs to broadcast only the average value. Each agent will then be able to compute $f(\cdot,x^{- i}_k)$ by subtracting from the average the value of its local decision vector at iteration $k$, i.e., $x^i_k$. The reader is referred to the case study of Section \ref{sec:secV} for an application that exhibits this structure.
\end{remark}

\section{Main convergence result} \label{sec:secIII}

In this section we analyze Algorithm \ref{alg:Alg1} and show that, for an appropriate choice of the regularization coefficient $c$, the algorithm converges to a minimizer of $\mathcal{P}$.

We start defining some matrices that will be used in the following:
for all $i=1,\ldots,m$, let $Q_{i,i}$ denote the $i$-th block of $Q$, with row and column indices corresponding to $x^i$, where $x = [x^{1,\top}\  \ldots\ x^{m,\top}]^\top$. Denote then by $Q_d$ a block diagonal matrix whose $i$-th block is $Q_{i,i}$, and let $Q_z = Q - Q_d$ denote the off (block) diagonal part of $Q$. Since $Q$ is assumed to be symmetric, $Q_z$ is symmetric as well and its eigenvalues are all real. Since $Q_z$ has zero trace, at least one of its eigenvalues will be non-negative. As a result, $\lambda^{\max}_{Q_z} \geq 0$, where $\lambda^{\max}_{Q_z}$ denotes the maximum eigenvalue of $Q_z$.

We are now in a position to state one of the main results of this paper.
\begin{thm} \label{thm:Alg_conv_quad_NE}
Under Assumption \ref{ass:Convexity}, if $c > \lambda^{\max}_{Q_z}$, then Algorithm \ref{alg:Alg1} converges
to a minimizer of $\mathcal{P}$.
\end{thm}

Theorem \ref{thm:Alg_conv_quad_NE} provides an explicit bound on $c$ that ensures convergence. Such a bound is derived by a fixed-point theoretical approach.
The preliminary results in Section \ref{sec:preliminary} are instrumental to the proof of Theorem \ref{thm:Alg_conv_quad_NE} in Section \ref{sec:proof_main}, and hold for a more general class of objective functions than the quadratic ones in Assumption \ref{ass:Convexity}. Finally, in Section \ref{sec_gradient} we show that Algorithm \ref{alg:Alg1} can be reinterpreted as a step of a scaled projected gradient algorithm, derive a bound on $c$ for convergence to some minimizer based on this reinterpretation, and show that the bound provided in Theorem \ref{thm:Alg_conv_quad_NE} is tighter.

\subsection{Preliminary results} \label{sec:preliminary}

In this section we define appropriate mappings and establish connections between the set of minimizers of the optimization problem with separable constraint $\mathcal{P}$ and the set of fixed-points of those mappings.
Results hold for the class of objective functions specified in Assumption \ref{ass:Convexity_bis}, which includes the quadratic objective functions in Assumption \ref{ass:Convexity}.

\begin{assumption}\label{ass:Convexity_bis}
The function $f: \mathbb{R}^{n_1} \times \ldots \times \mathbb{R}^{n_m} \to \mathbb{R}$ is continuously differentiable, and jointly convex with respect to all arguments. Moreover, the sets $X^i \subseteq \mathbb{R}^{n_i}$, $i=1,\ldots,m$, are non-empty, compact and convex.
\end{assumption}

\subsubsection{Minimizers and fixed-points definitions} \label{sec:secIIIA}

By \eqref{eq:Pobj}-\eqref{eq:Pcon}, the set of minimizers of $\mathcal{P}$ is given by
\begin{align}
M = \arg\min_{\{z^i \in X^i\}_{i=1}^m} f(z^1,\ldots,z^m) \subseteq X. \label{eq:Opt}
\end{align}
Following the discussion below Assumption \ref{ass:Convexity}, $M$ is non-empty. Note that $M$ is not necessarily a singleton; this will be the case if $f$ is jointly strictly convex with respect to its arguments.


For each $i$, $i=1,\ldots,m$, consider the mappings $T^i:X \to X^i$ and $\widetilde{T}^i: X \to X^i$, defined such that, for any $x = (x^1,\ldots,x^m) \in X$,
\begin{align}
T^{i}(x)= &\arg\min_{\zbf^i\in X^i}  \|z^i -x^i\|^2  \label{eq:T2i}\\
&\text{subject to } \nonumber \\
&\ f(z^i,x^{-i}) \le \min_{\zeta^i \in X^i} f(\zeta^i,x^{-i}) \nonumber\\
\widetilde{T}^i(x) =& \arg\min_{z^i \in X^i} \left\{f(z^i,x^{-i}) + c \|z^i - x^i\|^2\right\}. \label{eq:Ttilde}
\end{align}
The mapping in \eqref{eq:T2i} serves as a tie-break rule to select, in case $f(\cdot,x^{-i})$ admits multiple minimizers over $X^i$, the one closer to $x^i$ with respect to the Euclidean norm. Note that, with a slight abuse of notation, in \eqref{eq:T2i} and \eqref{eq:Ttilde} we use equality instead of inclusion since the corresponding minimizers $T^{i}(x)$ and $\widetilde{T}^i(x)$, respectively, are unique.
Note also that with $x_k$ in place of $x$, \eqref{eq:Ttilde} implies that the update step 6 in Algorithm \ref{alg:Alg1} can be equivalently represented by $x^i_{k+1} = \widetilde{T}^i(x_k)$.

Define also the mappings $T: X \to X$ and $\widetilde{T}: X \to X$, such that their components are given by $T^i$ and $\widetilde{T}^i$, respectively, for $i=1,\ldots,m$, i.e., $T= \big ( T^1,\ldots,T^m \big)$ and $\widetilde{T} = \big ( \widetilde{T}^1,\ldots,\widetilde{T}^m \big)$. The mappings $T$ and $\widetilde{T}$ can be equivalently written as
\begin{align}
T(x)=& \arg \min_{\zbf \in X} \sum_{i=1}^m \|z^i - x^i\|^2 \label{eq:T2} \\&\text{subject to} \nonumber \\ &  f(z^i,x^{-i}) \le \min_{\zeta^i \in X^i} f(\zeta^i,x^{-i}), ~\forall i=1,\ldots,m \nonumber\\
\widetilde{T}(x)=&\arg\min_{\zbf \in X} \sum_{i=1}^m \left\{f(z^i,x^{- i}) + c\|z^i-x^i\|^2\right\},\label{eq:Tmapsum}
\end{align}
where the terms inside the summation in \eqref{eq:T2} and \eqref{eq:Tmapsum} are decoupled.
The set of fixed-points of $T$ and $\widetilde{T}$ is, respectively, given by
\begin{align}
F_T &= \big \{ x \in X:~ x=T(x)\big \}, \\
F_{\widetilde{T}} &= \big \{ x \in X:~ x=\widetilde{T}(x)\big \},
\end{align}
and can be equivalently characterized in terms of the individual components of $T$ and $\widetilde{T}$, as
\begin{align}
F_T &= \big \{ x \in X:~ x^i=T^i(x), \text{ for all } i=1,\ldots,m \big \}, \label{eq:FP_T} \\
F_{\widetilde{T}} &= \big \{ x \in X:~ x^i=\widetilde{T}^i(x), \text{ for all } i=1,\ldots,m \big \}, \label{eq:FP_Ttilde}
\end{align}
to facilitate the subsequent derivations.

\subsubsection{Connections between minimizers and fixed-points} \label{sec:secIIIB}
We report here a fundamental optimality result, that we will often use in the sequel.

\begin{proposition}[{\cite[Proposition 3.1]{Bertsekas_Tsitsiklis_1997}}] \label{prop:opt_VI} Under Assumption \ref{ass:Convexity_bis},
\begin{enumerate}
\item if $x \in X$ minimizes $f$ over $X$, then $(z-x)^\top \nabla f(x) \geq 0$, for all $z \in X$.
\item if $f$ is also convex on $X$, then the condition of the previous part is also sufficient for $x$ to minimize $f$ over $X$, i.e., $x \in \arg \min_{z \in X} f(z)$.
\end{enumerate}
\end{proposition}

We start by showing that the set of minimizers $M$ of $\mathcal{P}$ in \eqref{eq:Opt} and the set of fixed-points $F_T$ of the mapping $T$ in \eqref{eq:T2} coincide. This is summarized in the following proposition.
\begin{proposition}\label{prop:Opt_FT}
Under Assumption \ref{ass:Convexity_bis}, $M = F_T$.
\end{proposition}
\begin{proof}
1) $M \subseteq F_T$: Fix any $x \in M$. For each $i=1,\ldots,m$, denote $x$ by $(x^i,x^{-i})$. The fact that $x \in M$ implies that
$(x^i,x^{-i})$ is a minimizer of $f$, for all $i=1,\ldots,m$. As a result, $f(x^i,x^{-i})$ will be no greater than the values that $f$ may take over $X$, and hence also for any values it may take if evaluated at $(\zeta^i,x^{-i})$, for any $\zeta^i \in X^i$, i.e.,
$f(x^i,x^{-i}) \leq f(\zeta^i,x^{-i})$, for all $\zeta^i \in X^i$.
The last statement can be equivalently written as
\begin{align}
f(x^i,x^{-i}) \leq \min_{\zeta^i \in X^i} f(\zeta^i,x^{-i}), \label{eq:OptNash_pf2}
\end{align}
which means that $x$ satisfies the inequality in \eqref{eq:T2}. Moreover $x$ is also optimal for the objective function in \eqref{eq:T2}, since it results in zero cost. Hence, by \eqref{eq:T2}, $x$ is a fixed-point of $T$, which, by \eqref{eq:FP_T}, implies that $x \in F_T$, thus concluding the first part of the proof.\\
2) $F_T \subseteq M$: Fix any $x \in F_T$. By the definition of $F_T$, and due to the inequality in \eqref{eq:T2} that is embedded in the definition of $T$, we have that for all $i=1,\ldots,m$, $f(x^i,x^{-i}) \leq \min_{\zeta^i \in X^i} f(\zeta^i,x^{-i})$. The last statement implies that, for all $i=1,\ldots,m$, $x^i$ is the minimizer of $f(\cdot,x^{-i})$ over $X^i$. For all $i=1,\ldots,m$, by the first part of Proposition \ref{prop:opt_VI} (with $f(\cdot,x^{-i})$ in place of $f$) we then have that
\begin{align}
(z^i-x^i)^\top \nabla^i f(x^i,x^{-i}) \geq 0, \text{ for all } z^i \in X^i, \label{eq:OptNash_pf3}
\end{align}
where $\nabla^i f(x^i,x^{-i})$ is the $i$-th component of the gradient $\nabla f(\cdot,x^{-i})$ of $f(\cdot,x^{-i})$, evaluated at $x^i$.
By \eqref{eq:OptNash_pf3}, we then have that $\sum_{i=1}^m (z^i-x^i)^\top \nabla^i f(x^i,x^{-i}) \geq 0$ for all $z^i \in X^i$, $i=1,\ldots,m$, which, by setting $x = (x^1,\ldots,x^m)$, $z = (z^1,\ldots,z^m)$, can be written as $(z-x)^\top \nabla f(x) \geq 0$, for all $z \in X$. By the second part of Proposition \ref{prop:opt_VI}, and since $f$ is jointly convex with respect to all elements of $x$, the last statement implies that $x$ is a minimizer of $f$ over $X$, i.e., $x \in M$, thus concluding the second part of the proof.
\end{proof}

Note that the connection between minimizers, fixed-points and variational inequalities, similar to the ones that appear in the proof of Proposition \ref{prop:Opt_FT} (e.g., see \eqref{eq:OptNash_pf3}), has been also investigated in \cite{Facchinei_etal_2007}, in the context of Nash equilibria in non-cooperative games.

We next show that the set of fixed-points $F_T$ of $T$ and the set of fixed-points $F_{\widetilde{T}}$ of $\widetilde{T}$ coincide. This is summarized in the following proposition.

\begin{proposition}\label{prop:FP_T_TT}
Under Assumption \ref{ass:Convexity_bis}, $F_T =  F_{\widetilde{T}}$.
\end{proposition}
\begin{proof}
1) $F_T \subseteq F_{\widetilde{T}}$: Fix any $x \in F_T$. By \eqref{eq:FP_T}, this is equivalent to the fact that $x^i = T^i(x)$, for all $i=1,\ldots,m$, which, due to the definition of $T$ implies that, for all $i=1,\ldots,m$,
\begin{align}
f(x^i,x^{-i}) \le \min_{\zeta^i \in X^i} f(\zeta^i,x^{-i}). \label{eq:FP_FP_pf1-2}
\end{align}
This implies that $x^i$ minimizes $f(\cdot,x^{-i})$ over $X^i$, hence, by the first part of Proposition \ref{prop:opt_VI} (with $f(\cdot,x^{-i})$ in place of $f$) we have that
\begin{align}
(z^i-x^i)^\top \nabla^i f(x^i,x^{-i}) \geq 0, \text{ for all } z^i \in X^i, \label{eq:FP_FP_pf2-2}
\end{align}
where $\nabla^i f(x^i,x^{-i})$ is the $i$-th component of the gradient $\nabla f(\cdot,x^{-i})$ of $f(\cdot,x^{-i})$, evaluated at $x^i$.

Let $f_c(z^i,x) = f(z^i,x^{-i}) + c \|z^i-x^i\|^2$, for all $z^i$, $i=1,\ldots,m$, and notice that $\nabla f_c(x^i,x) = \nabla f(x^i,x^{-i})$, where $\nabla f_c(x^i,x)$ is the gradient of $f_c(\cdot,x)$, evaluated at $x^i$. The latter is due to the fact that the gradient of the quadratic penalty term vanishes at $x^i$. By \eqref{eq:FP_FP_pf2-2}, we then have that, for all $i=1,\ldots,m$,
\begin{align}
(z^i-x^i)^\top \nabla^i f_c(x^i,x) \geq 0, \text{ for all } z^i \in X^i. \label{eq:FP_FP_pf3-2}
\end{align}
Since $f_c(\cdot,x)$ is strictly convex with respect to its first argument, by the second part of Proposition \ref{prop:opt_VI} (with $f_c(\cdot,x)$ in place of $f$), \eqref{eq:FP_FP_pf3-2} implies that, for all $i=1,\ldots,m$, $x^i$ is the unique minimizer of $f_c(\cdot,x)$ over $X^i$, i.e.,
\begin{align}
x^i = \arg \min_{z^i \in X^i} f(z^i,x^{-i}) + c \|z^i-x^i\|^2. \label{eq:FP_FP_pf4-2}
\end{align}
By \eqref{eq:Ttilde}, \eqref{eq:FP_FP_pf4-2} is equivalent to $x^i = \widetilde{T}^i(x)$, for all $i=1,\ldots,m$, thus concluding the first part of the proof.\\
2) $F_{\widetilde{T}} \subseteq F_T$: Fix any $x \in F_{\widetilde{T}}$. By \eqref{eq:FP_Ttilde} this is equivalent to the fact that $x^i = \widetilde{T}^i(u)$, for all $i=1,\ldots,m$, which, by the definition of $\widetilde{T}^i$ in \eqref{eq:Ttilde}, implies that, for all $i=1,\ldots,m$,
\begin{align}
x^i = \arg \min_{z^i \in X^i} f(z^i,x^{-i}) + c \|z^i-x^i\|^2. \label{eq:FP_Nash_pf1-2}
\end{align}
Let again $f_c(z^i,x) = f(z^i,x^{-i}) + c \|z^i-x^i\|^2$. Equation \eqref{eq:FP_Nash_pf1-2} implies then that, for all $i=1,\ldots,m$, $x^i$ minimizes $f_c(\cdot,x)$ over $X^i$, and by the first part of Proposition \ref{prop:opt_VI} (with $f_c(\cdot,x)$ in place of $f$) leads to
\begin{align}
(z^i-x^i)^\top \nabla^i f_c(x^i,x) \geq 0, \text{ for all } z^i \in X^i, \label{eq:FP_Nash_pf2-2}
\end{align}
where $\nabla^i f_c(x^i,x)$ is the $i$-th component of the gradient $\nabla f_c(\cdot,x)$ of $f_c(\cdot,x)$, evaluated at $x^i$.

Notice that $\nabla f_c(x^i,x) = \nabla f(x^i,x^{-i})$, where $\nabla f(x^i,x^{-i})$ is the gradient of $f(\cdot,x^{-i})$, evaluated at $x^i$, since the gradient of $\|z^i-x^i\|^2$ with respect to $z^i$ vanishes at $x^i$. Therefore, for all $i=1,\ldots,m$, \eqref{eq:FP_Nash_pf2-2} leads to
\begin{align}
(z^i-x^i)^\top \nabla^i f(x^i,x^{-i}) \geq 0, \text{ for all } z^i \in X^i, \label{eq:FP_Nash_pf3-2}
\end{align}
Since $f(\cdot,x^{-i})$ is convex with respect to its first argument, by the second part of Proposition \ref{prop:opt_VI}, \eqref{eq:FP_Nash_pf3-2} implies that $x^i$ minimizes $f(\cdot,x^{-i})$ over $X^i$. In other words, $x^i \in \arg \min_{z^i \in X^i} f(z^i,x^{-i})$, for all $i=1,\ldots,m$. This in turn implies that, for all $i=1,\ldots,m$,
\begin{align}
f(x^i,x^{-i}) &\leq f(z^i,x^{-i}), \text{ for all } z^i \in X^i, \nonumber \\
\Leftrightarrow ~f(x^i,x^{-i}) &\leq \min_{z^i \in X^i} f(z^i,x^{-i}).
\end{align}
The last inequality shows that $u$ satisfies the inequality in \eqref{eq:T2} in the definition of $T$. Moreover, it minimizes the objective function in \eqref{eq:T2}, since it results in zero cost. Therefore, we have $x=T(x)$, thus concluding the second part of the proof.
\end{proof}

A direct consequence of Propositions \ref{prop:Opt_FT} and \ref{prop:FP_T_TT} is that the set of minimizers $M$ of $\mathcal{P}$ coincides with the set of fixed-points of the mapping $\widetilde{T}$. This is summarized in the following corollary.
\begin{cor} \label{cor:M_FP_Ttilde}
Under Assumption \ref{ass:Convexity_bis}, $M=  F_{\widetilde{T}}$.
\end{cor}

\subsection{Proof of Theorem \ref{thm:Alg_conv_quad_NE}} \label{sec:proof_main}

Step 6 of Algorithm \ref{alg:Alg1} can be equivalently written as $x^i_{k+1} = \widetilde{T}^i (x_{k})$, which entails that
$x_{k+1} = \widetilde{T} (x_{k})$, i.e., one iteration of Algorithm \ref{alg:Alg1} for the multi-agent system corresponds to a Picard-Banach iteration of the mapping $\widetilde{T}$ (see \cite{Berinde} (Chapter 1.2) for a definition).

Since the set of fixed-points of $\widetilde{T}$ is non-empty (it coincides with $M$ due to Corollary \ref{cor:M_FP_Ttilde}), we just need to prove that $\widetilde{T}$ is firmly non-expansive (see \cite{Combettes_Pennanen_2002} (Section 1) for a definition in general Hilbert spaces). If that is the case, then, by the results of \cite{Combettes_Pennanen_2002,Opial_1967}, we have that the Picard-Banach iteration converges to a fixed-point of $\widetilde{T}$, for any initial condition $x^i_0 \in X^i$, $i=1,\ldots,m$. By Corollary \ref{cor:M_FP_Ttilde} this fixed-point will also be a minimizer of $\mathcal{P}$.

We next show that if $c > \lambda^{\max}_{Q_z}$, then, the mapping $\widetilde{T}(\cdot)$ is indeed firmly non-expansive  with respect to $\| \cdot \|_{Q_d + I_c - Q}$, i.e.,
\begin{align}
\| \widetilde{T}(x) &- \widetilde{T}(y)\|^2_{Q_d + I_c - Q}  \nonumber \\ &\leq (x-y)^\top (Q_d + I_c - Q) (\widetilde{T}(x) - \widetilde{T}(y)), \label{eq:firmlyNE}
\end{align}
thus concluding the proof of Theorem \ref{thm:Alg_conv_quad_NE}.

Under Assumption \ref{ass:Convexity}, the mapping $\widetilde{T}$ in \eqref{eq:Tmapsum} is given by
\begin{align}
\widetilde{T}(x)&=\arg\min_{\zbf \in X} \sum_{i=1}^m f(z^i,x^{- i}) + c\|z^i-x^i\|^2 \nonumber\\
&=\arg\min_{\zbf \in X} \sum_{i=1}^m (z^i)^\top(Q_{i,i} + I_c)z^i \nonumber \\ &~~~~~~~~~~~~~+ (2(x^{- i})^\top Q_{- i,i} -2(x^i)^\top I_c +q_i^\top)z^i  \nonumber \\
&=\arg\min_{\zbf \in X} z^\top (Q_d+I_c) z + (2x^\top Q_z -2x^\top I_c +q^\top)z. \label{eq:map_T_tilde_quad}
\end{align}
Notice the slight abuse of notation in \eqref{eq:map_T_tilde_quad}, where the weighted identity matrix $I_c$ in the second and the third equality are not of the same dimension.
Let $\csbf(x)=(Q_d+I_c)^{-1}(I_cx-Q_{z}x-q/2)$ denote the unconstrained minimizer of \eqref{eq:map_T_tilde_quad}. We then have that
\begin{align}
\widetilde{T}(x)&=\arg\min_{\zbf \in X} (z -\csbf(x))^\top(Q_{d} + I_c)(z -\csbf(x)) \nonumber \\ &=\left[ \csbf(x) \right]_{Q_d+I_c}^{X}, \label{eq:TmapProj}
\end{align}
where $\left[ \csbf(x) \right]_{Q_d+I_c}^{X}$ denotes the projection, with respect to $||\cdot||_{Q_d + I_c}$, of $\csbf(x)$ on $X$.
Note that $Q_d+I_c$ is always positive definite for $c \in \mathbb{R}_+$, so that its inverse exists, and the projection is well defined.

We have that
\begin{align}
&\|\widetilde{T}(x)-\widetilde{T}(y)\|_{Q_d+I_c}^2 \nonumber\\
& =\|\left[ \csbf(x) \right]_{Q_d+I_c}^{X} - \left[ \csbf(y) \right]_{Q_d+I_c}^{X}\|_{Q_d+I_c}^2\nonumber\\
& \le(\csbf(x) - \csbf(y))^\top(Q_d+I_c)(\left[ \csbf(x) \right]_{Q_d+I_c}^{X} - \left[ \csbf(y) \right]_{Q_d+I_c}^{X}) \nonumber \\
& =(x-y)^\top(I-Q(Q_d+I_c)^{-1})(Q_d+I_c) \nonumber \\
& ~~~~~~~~~~~\times (\left[ \csbf(x) \right]_{Q_d+I_c}^{X} - \left[ \csbf(y) \right]_{Q_d+I_c}^{X})\nonumber\\
& =(x-y)^\top(Q_d+I_c-Q)(\left[ \csbf(x) \right]_{Q_d+I_c}^{X} - \left[ \csbf(y) \right]_{Q_d+I_c}^{X}), \label{eq:condFNE1}
\end{align}
where the first inequality follows from the definition of a firmly non-expansive mapping and the fact that any projection mapping is firmly non-expansive (see Proposition 4.8 in \cite{combettes2010}).
The second equality is due to the definition $\csbf$, and the last one follows after performing the matrix multiplication.

Since $Q \succeq 0$, then $\|\widetilde{T}(x)-\widetilde{T}(y)\|_{Q_d+I_c-Q}^2 \leq \|\widetilde{T}(x)-\widetilde{T}(y)\|_{Q_d+I_c}^2$. This, together with \eqref{eq:condFNE1}, implies that
\begin{align}
&\|\left[ \csbf(x) \right]_{Q_d+I_c}^{X} - \left[ \csbf(y) \right]_{Q_d+I_c}^{X}\|_{Q_d+I_c-Q}^2\nonumber\\
&\le (x-y)^\top(Q_d+I_c-Q)(\left[ \csbf(x) \right]_{Q_d+I_c}^{X} - \left[ \csbf(y) \right]_{Q_d+I_c}^{X}). \label{eq:condFNE2}
\end{align}
By the definition of a firmly non-expansive mapping \cite{combettes2010}, \eqref{eq:condFNE2} implies that, if $Q_d+I_c-Q\succ0$,
$\widetilde{T}$ is firmly non-expansive with respect to $||\cdot||_{Q_d + I_c -Q}$. The condition $Q_d+I_c-Q\succ0$ can be satisfied by choosing $c > \lambda^{\max}_{Q_z}$.

\subsection{Connection with gradient algorithms}\label{sec_gradient}

%


Recalling the formulation in \eqref{eq:map_T_tilde_quad} and \eqref{eq:TmapProj},  
$x^i_{k+1} = \widetilde{T}^i (x_{k})$, $i=1,\ldots,m$, in step 6 of Algorithm \ref{alg:Alg1} can be equivalently written as a scaled projected gradient step as follows:
\begin{align}
&x_{k+1}=\left[\xi(x_k)\right]^X_{Q_d+I_c} \nonumber \\
&~~=\left[(Q_d+I_c)^{-1}(Q_d+I_c-Q)x_k - (Q_d+I_c)^{-1}\dfrac{q}{2}\right]^X_{Q_d+I_c} \nonumber \\
&~~=\left[x_k - (Q_d+I_c)^{-1}(Qx_k+\dfrac{q}{2})\right]^X_{Q_d+I_c} \nonumber \\
&~~=\left[x_k - \dfrac{1}{2c}(\dfrac{Q_d}{c}+I)^{-1}(2Qx_k+q)\right]^X_{\frac{Q_d}{c}+I} \label{eq:scaled_grad}
\end{align}
where, the first equality follows recalling the definition of $\xi(x)$ and of $Q_z$, and the last equality is obtained scaling by $c$. As one can see the gradient $2Qx+q$ of the original cost appears from the definition of $\xi(x)$, $1/(2 c)$ plays the role of the gradient step-size, and $(\dfrac{Q_d}{c} + I)$ is the scaling matrix (see  \cite[Section 3.3.3]{Bertsekas_Tsitsiklis_1997}).

Notice that $Q_d$ is symmetric with $Q_d \succeq 0$, as a result of $Q$ having the same property. Therefore, for any $c > 0$, the scaling matrix $(\dfrac{Q_d}{c} + I)$ satisfies the positivity condition
\begin{align}
(x-y)^\top (\dfrac{Q_d}{c} + I) (x-y) \geq \|x-y\|_2^2, \text{ for all } x, y \in X. \label{eq:scaled_grad_pos}
\end{align}
Under \eqref{eq:scaled_grad_pos}, by Proposition 3.7, p. 217 of \cite{Bertsekas_Tsitsiklis_1997} we have that the scaled projected gradient iteration \eqref{eq:scaled_grad} converges to some minimizer of $\mathcal{P}$ for a sufficiently small step-size $1/(2 c)$. This step-size is, however, not quantified in \cite{Bertsekas_Tsitsiklis_1997}; we perform this in the sequel since it provides the means to compare our methodology with a scaled projected gradient algorithm.

We can write $x_{k+1}$ in \eqref{eq:scaled_grad} as the unique solution to the following quadratic minimization program:
\begin{align}
x_{k+1} = \arg \min_{z \in X} c (z - x_k)^\top &(\dfrac{Q_d}{c} + I) (z - x_k) \nonumber \\ &+ c (z - x_k)^\top (2Qx_k+q).
\end{align}
By optimality of $x_{k+1}$, and since $z=x_k \in X$ results in  zero objective value, we have that
\begin{align}
& c (x_{k+1} - x_k)^\top (\dfrac{Q_d}{c} + I) (x_{k+1} - x_k) \nonumber \\ &~~~~~~~+ (x_{k+1} - x_k)^\top (2Qx_k+q) \leq 0, \nonumber \\
\Leftrightarrow ~~& (x_{k+1} - x_k)^\top (2Qx_k+q) \nonumber \\ &~~~~~~~\leq - c (x_{k+1} - x_k)^\top (\dfrac{Q_d}{c} + I) (x_{k+1} - x_k). \label{eq:scaled_opt}
\end{align}
By \eqref{eq:scaled_grad_pos} and \eqref{eq:scaled_opt} (notice that $x_{k+1}, x_k \in X$), we thus have that
\begin{align}
(x_{k+1} - x_k)^\top (2Qx_k+q) \leq - c \|x_{k+1} - x_k\|_2^2. \label{eq:scaled_opt1}
\end{align}

By the Descent Lemma (Lemma 2.1 in \cite{Bertsekas_Tsitsiklis_1997}), for the quadratic objective function of Assumption \ref{ass:Convexity} we obtain that
\begin{align}
f(x_{k+1}) \leq f(x_k) &+ (x_{k+1} - x_k)^\top (2Q x_k + q) \nonumber \\ &+ \lambda_{Q}^{\max} \|x_{k+1} - x_k\|_2^2, \label{eq:descent_lemma}
\end{align}
where $\lambda_{Q}^{\max}$ denotes the maximum eigenvalue of $Q$, which equals half of the Lipschitz constant of the gradient of $f$.
By \eqref{eq:descent_lemma} and \eqref{eq:scaled_opt1} we then have that
\begin{align}
f(x_{k+1}) \leq f(x_k) - (c - \lambda_{Q}^{\max}) \|x_{k+1} - x_k\|_2^2, \label{eq:descent_cond}
\end{align}

If $c > \lambda_{Q}^{\max}$, \eqref{eq:descent_cond} implies that $f(x_{k+1}) \leq f(x_k)$. Based on this montonicity condition, and following the proof of Proposition 3.3, p. 214, of \cite{Bertsekas_Tsitsiklis_1997} for the unscaled gradient method, it can be then shown that \eqref{eq:scaled_grad} is a convergent iteration to some minimizer of $\mathcal{P}$.

\subsubsection{Comparison with the scaled projected gradient algorithm}
In the scaled projected gradient algorithm it was shown that the step-size should be chosen so that $c>\lambda_{Q}^{\max}$, where $\lambda^{\max}_Q$ denotes the maximum eigenvalue of $Q$. Instead, Theorem \ref{thm:Alg_conv_quad_NE} requires $c>\lambda_{Q_z}^{\max}$. This latter is a less restrictive condition since $\lambda_{Q}^{\max}\ge\lambda_{Q_z}^{\max}$. Indeed, let  $v$ be the eigenvector corresponding to the eigenvalue $\lambda_{Q_z}^{\max}$. Then, we have
\begin{align}
\lambda_{Q_z}^{\max} v^\top v &=v^\top Q_z v =v^\top (Q-Q_d) v \nonumber \\ &=v^\top Q v -v^\top Q_d v
\le v^\top Q v,
\end{align}
where the first equality follows from the fact that $v$ is the eigenvector corresponding to $\lambda_{Q_z}^{\max}$ and the last inequality follows from $Q_d \succeq 0$. This implies that
\begin{align}
&\lambda_{Q_z}^{\max} \le \frac{v^\top Q v}{v^\top  v} \le \max_{z\neq0} \frac{z^\top Q z}{z^\top  z}=\lambda_{Q}^{\max},
\end{align}
where the last equality follows recalling the definition of the induced 2-norm of a symmetric square matrix.

\subsubsection{Dependence of the regularization coefficient $c$ on the number of agents}
We next provide some examples on how $\lambda_{Q_z}^{\max}$ and $\lambda_{Q}^{\max}$ are affected by the structure of the matrix $Q$ and by the number of agents $m$. As one can expect, the more the diagonal part of $Q$ is dominant with respect to the off diagonal part, the more the difference between $\lambda_{Q_z}^{\max}$ and $\lambda_{Q}^{\max}$ becomes significant: this is shown in Table \ref{tab:Q} for some choices of $Q$ (we assume $n_i=1$, $i=1,\ldots,m$, for simplicity). Matrix $\mathbf{1}_{m\times m}$ is an $m \times m$ matrix with all its elements being equal to $1$. \\
\begin{table}[t]
  \centering
  \caption{Values of $\lambda_{Q_z}^{\max}$ and $\lambda_{Q}^{\max}$ for different choices of $Q$.}
    \begin{tabular}{l|c|c}
    \toprule
$Q$ &   $\lambda_{Q_z}^{\max}$ & $\lambda_{Q}^{\max}$  \\
    \midrule \midrule
$\mathbf{1}_{m\times m}$ & $m-1$ & $m$\\
$\mathbf{1}_{m\times m} + mI_{m\times m}$ & $m-1$ & $2m$\\
$mI_{m\times m} $ & $0$ & $m$\\
    \bottomrule
    \end{tabular}%
  \label{tab:Q}%
\end{table}%

Note that the condition on the regularization coefficient $c$ provided in Theorem \ref{thm:Alg_conv_quad_NE} implicitly accounts through $\lambda_{Q_z}^{\max}$  for how much the agents decisions are coupled in the objective function.

\section{Extension to differentiable convex cost functions} \label{sec:secIV}
In this section, we provide a convergence result for Algorithm \ref{alg:Alg1} considering a more general setting as in Assumption \ref{ass:Convexity_bis} under the additional differentiability and Lipschitz continuity assumptions below.

\begin{assumption} \label{ass:grad_Lip}
The cost function $f:\mathbb{R}^n \to \mathbb{R}$ is continuously differentiable.
The gradient $\nabla f: \mathbb{R}^n \to \mathbb{R}^n$ of $f: \mathbb{R}^n \to \mathbb{R}$ is Lipschitz continuous on $X \subset \mathbb{R}^n$ with Lipschitz constant $L \in \mathbb{R}_+$, i.e., for all $x, y \in X$,
\begin{align} \label{eq:grad_Lip}
\| \nabla f(x) - \nabla f(y) \| \leq L \|x-y\|.
\end{align}
\end{assumption}

The results in the following Lemma \ref{lm:Lip_constant} and Theorem \ref{thm:Alg_conv} are obtained relying on a different approach with respect to the one based on fixed-points and map properties exploited in the proof of Theorem \ref{thm:Alg_conv_quad_NE}. Their proofs are reported in the Appendix.

\begin{lemma}\label{lm:Lip_constant}
Under Assumptions \ref{ass:Convexity_bis} and \ref{ass:grad_Lip}, for all $x,y,z \in X \subset \mathbb{R}^n$, with $x = (x^1,\ldots,x^m)$, $y = (y^1,\ldots,y^m)$, and $z = (z^1,\ldots,z^m)$,
\begin{align}\label{eq:Lip_constant}
\|\nabla\Big(\sum_{i=1}^m f(\cdot,x^{-i}) \Big ) \Big|_{z} &- \nabla\Big(\sum_{i=1}^m f(\cdot,y^{-i}) \Big ) \Big|_{z}\| \nonumber \\ &\le \sqrt{m} L \|x-y\|,
\end{align}
where $L \in \mathbb{R}_+$ is the Lipschitz constant of $\nabla f$.
\end{lemma}

\begin{thm} \label{thm:Alg_conv}
Consider Algorithm \ref{alg:Alg1}. Under Assumptions \ref{ass:Convexity_bis} and \ref{ass:grad_Lip}, if
\begin{align}
c > \frac{m-1}{2m-1} \sqrt{m}L, \label{eq:coef_c}
\end{align}
then $\lim_{k \to \infty} \mathrm{dist}(x_k,X^*) = 0$, where $X^*$ is the set of minimizers of $\mathcal{P}$.
\end{thm}

Note that the iterates generated by Algorithm \ref{alg:Alg1} may not necessarily converge to a minimizer of $\mathcal{P}$, since $\{x_k\}_{k \geq 0}$ may exhibit an oscillatory behavior; however, all its limit points achieve the optimal value.

By inspecting \eqref{eq:coef_c}, it can be observed that as the number of agents $m$ tends to infinity, $c$ tends to infinity as well, since $\lim_{m \to +\infty} \frac{m-1}{2m-1} \sqrt{m}L = +\infty$. This stems from the fact that Lipschitz constant in Lemma \ref{lm:Lip_constant} depends on $m$. This fact implies that the optimization problem at step 6 of Algorithm \ref{alg:Alg1} tends to be numerically ill-conditioned as the number of agents tends to infinity; but this is also the case for the centralized problem $\mathcal{P}$.
It should be also noted that in the other extreme situation of a single agent ($m=1$), \eqref{eq:coef_c} implies that any $c > 0$ is sufficient for the assertion of Theorem \ref{thm:Alg_conv} to hold. This is in line with
\cite{Bertsekas_Tsitsiklis_1997} (Proposition 4.1, p. 233),
since step 6 of Algorithm \ref{alg:Alg1} with $m=1$, reduces to the standard proximal minimization.
The fact that the higher $m$, the higher $c$ needs to be for the statement of Theorem \ref{thm:Alg_conv} to hold, can be intuitively justified by the fact that for any $i$, $i=1,\ldots,m$, a higher $c$ will slow down the progress of the iterates $\{x^i(k)\}_{k\geq0}$, of agent $i$ towards its optimal value, so that all agents get ``synchronized'' in optimizing the common objective function with respect to their local decision vectors.

The approach adopted in Lemma \ref{lm:Lip_constant} and Theorem \ref{thm:Alg_conv} can also be applied when the cost function $f$ is quadratic, since  Assumptions \ref{ass:Convexity_bis} and \ref{ass:grad_Lip} are both satisfied under Assumption \ref{ass:Convexity}. The following Lemma applies to that case, and makes the Lipschitz constant $\sqrt{m}L$ in \eqref{eq:Lip_constant} explicit as a function of $\lambda^{\max}_{Q_z}$ (see the Appendix for a proof).

\begin{lemma} \label{lm:Lip_constant_quad}
Consider the set-up of Lemma \ref{lm:Lip_constant}, and Assumption \ref{ass:Convexity}. The Lipschitz constant $\sqrt{m}L$ in \eqref{eq:Lip_constant} can be replaced by $2\lambda^{\max}_{Q_z}$, where $\lambda^{\max}_{Q_z}$ denotes the maximum eigenvalue of matrix $Q_z$.
\end{lemma}

As a consequence, Theorem \ref{thm:Alg_conv} rewrites as follows.

\begin{thm} \label{thm:Alg_conv_quad}
Consider Algorithm \ref{alg:Alg1}. Under Assumption \ref{ass:Convexity}, if
\begin{align}
c > \frac{m-1}{2m-1} 2\lambda^{\max}_{Q_z}, \label{eq:coef_c_quad}
\end{align}
then $\lim_{k \to \infty} \mathrm{dist}(x_k,X^*) = 0$, where $X^*$ is the set of minimizers of $\mathcal{P}$.
\end{thm}

Note that \eqref{eq:coef_c_quad} implies that as the number of agents $m$ tends to infinity, $c$ tends to $ \lambda^{\max}_{Q_z}$. The latter, however, does not necessarily imply that $c$ remains finite, since, in certain cases, $\lambda^{\max}_{Q_z}$ may tend to infinity.

The condition $c > \lambda^{\max}_{Q_z}$ of Theorem \ref{thm:Alg_conv_quad_NE} is sufficient for \eqref{eq:coef_c_quad} to be satisfied, since $\frac{1}{2} > \frac{m-1}{2m-1}$, for all $m$.
However, Theorem \ref{thm:Alg_conv_quad_NE} ensures convergence to some minimizer and not just convergence in optimal value as in Theorem \ref{thm:Alg_conv}.

\section{Optimal charging of electric vehicles} \label{sec:secV}
\subsection{Problem setup} \label{sec:secVA}
We consider the problem of optimizing the charging strategy for a fleet of $m$ plug-in electric vehicles (PEVs) over a finite horizon $T$. Following \cite{Grammatico_etal_2015,Franc_2014,Hiskens_2013}, the PEV charging problem is given by the following optimization problem.
\begin{align}
&\min_{\big \{ \{ x^{i}(t) \}_{i=1}^m \big \}_{t=0}^T} \frac{1}{m}\sum_{t=0}^{T} p(t)\Big ( d(t)+\sum_{i=1}^{m} x^{i}(t) \Big )^2 \label{pb:EVpb1}\\
&\text{subject to } \nonumber \\
&\sum_{t=0}^{T} x^{i}(t) = \gamma^{i}, \text{ for all } i=1,\ldots, m \nonumber \\
& \underline{x}^{i}(t) \le x^{i}(t) \le \overline{x}^{i}(t), \text{ for all } i=1,\ldots ,m,~ t=0,\ldots, T, \nonumber 
\end{align}
where $p(t) \in \mathbb{R}$ is an electricity price coefficient at time $t$, $d(t) \in \mathbb{R}$ represents the non-PEV demand at time $t$, $x^{i}(t) \in \mathbb{R}$ is the charging rate of vehicle $i$ at time $t$, $\gamma^i \in \mathbb{R}$ represents a prescribed charging level to be reached by each vehicle $i$ at the end of the considered time horizon, and $\underline{x}^{i}(t), \overline{x}^{i}(t) \in \mathbb{R}$ are bounds on the minimum and maximum value of $x^{i}(t)$, respectively. The objective function in \eqref{pb:EVpb1} encodes the total electricity cost given by the demand (both PEVs and non-PEVs) multiplied by the price of electricity, which in turn depends linearly on the total demand through $p(t)$, thus giving rise to the quadratic function in \eqref{pb:EVpb1}. This linear dependency of price with respect to the total demand models the fact that agents/vehicles are price anticipating authorities, anticipating their consumption to have an effect on the electricity price (see \cite{Gharesifard_etal_2016} for further elaboration).
Problem \eqref{pb:EVpb1} can be rewritten in a more compact form as
\begin{align}
&\min_{x \in \mathbb{R}^{m(T+1)}} (\dbf+Ax)^\top P(\dbf+Ax)\label{pb:EVpb2}\\
&\text{subject to } x^i \in X^i, \text{ for all } i=1,\ldots, m, \nonumber
\end{align}
where $P=(1/m) \mathrm{diag}(p) \in \mathbb{R}^{(T+1) \times (T+1)}$, and $\mathrm{diag}(p)$ is a matrix with $p = (p(0),\ldots,p(T)) \in \mathbb{R}^{T+1}$ on its diagonal. $A=\mathbf{1}_{1 \times m} \otimes I \in \mathbb{R}^{(T+1) \times m}$, where $\otimes$ denotes the Kronecker product, and $I$ is the identity matrix of appropriate dimension. Moreover, $d=(\dbf(0), \ldots, \dbf(T)) \in \mathbb{R}^{T+1}$, $x=(x^1, \ldots, x^{m}) \in \mathbb{R}^{m(T+1)}$ with $x^i=(x^{i}(0), \ldots, x^{i}(T)) \in \mathbb{R}^{T+1}$, and $X^i$ encodes the constraints of each vehicle $i$, $i=1,\ldots,m$, in \eqref{pb:EVpb1}.

Problem \eqref{pb:EVpb2} can be solved in a decentralized fashion by means of Algorithm \ref{alg:Alg1}. We compute the value of $c$ so as to satisfy \eqref{eq:coef_c_quad} and, hence, to ensure convergence of the algorithm.
Note that the objective function in \eqref{pb:EVpb2} is not strictly convex as $A^\top PA=\mathbf{1}_{m	 \times m} \otimes P$, and it exhibits a structure that allows for reduced information exchange as described in Remark 1. Indeed, at iteration $k+1$ of Algorithm \ref{alg:Alg1}, the central authority needs to collect the solution of each agent but it only has to broadcast $\bar{x}_k=\dbf+Ax_k$. Each agent $i$, $i=1,\ldots,m$, can then compute its objective as $f(\zbf^i,x^{-i}_k)=(\bar{x}_k-x^i_k+\zbf^i)^\top P(\bar{x}_k-x^i_k+\zbf^i)$. Step 6 in Algorithm \ref{alg:Alg1} for problem \eqref{pb:EVpb2} reduces then to
\begin{align*}
&x^i_{k+1}=\widetilde{T}^i(x_k) =\nonumber \\ & \arg\!\min_{\zbf^i\in X^i} \big \{ (\bar{x}_k-x^i_k+\zbf^i)^\top P(\bar{x}_k-x^i_k+\zbf^i) + c\|z^i-x_k^i\|^2\! \big \}.
\end{align*}

\subsection{Simulation results} \label{sec:secVB}
We consider first a fleet of $m=100$ PEVs, each of them having to reach a different level of charge $\gamma^i \in [0.1, 0.3]$, $i=1,\ldots,m$, at the end of a time horizon $T=24$, corresponding to hourly steps. The bounds on $x^i(t)$ are taken to be $\underline{x}^{i}(t)=0$ and $\overline{x}^{i}(t)=0.02$, for all $i=1,\ldots,m$, $t=0,\ldots,T$. The non-PEV demand profile is retrieved from \cite{Hiskens_2013}, whereas the price coefficient is $p(t)=0.15$, $t=0,\ldots,T$. Note that, as in \cite{Franc_2014}, $x^{i}(t)$ corresponds to normalized charging rate, which is then rescaled to be turned into reasonable power values.
All optimization problems are solved using CPLEX \cite{cplex}.

For comparison purposes, problem \eqref{pb:EVpb2} is solved first in a centralized fashion, achieving an optimal objective value $f^\star=2.67$. It is then solved in a decentralized fashion by means of Algorithm \ref{alg:Alg1}, setting $c=0.1485$. Note that for \eqref{eq:coef_c_quad} to be satisfied, we should have $c >0.1478$.
In Figure \ref{fig:Es1_f1a} the objective value $f(x_k)$ achieved at iteration $k$ of Algorithm \ref{alg:Alg1} is depicted, whereas Figure \ref{fig:Es1_f1b} shows $\|x_k-x_{k-1}\|$. After 30 iterations the difference between the decentralized and the centralized objective is $f(x_{30})-f^\star=1.95\cdot10^{-6}$, thus achieving numerical convergence.

\begin{figure}[t!]
\centering
\subfloat[Objective value $f(x_k)$ (red stars) at iteration $k$ of Algorithm \ref{alg:Alg1}, and \newline optimal value $f^\star$ (dashed line) of the centralized problem counterpart.]{\includegraphics[trim=2cm 8.25cm 2cm 14.25cm, clip=true,scale=0.55]{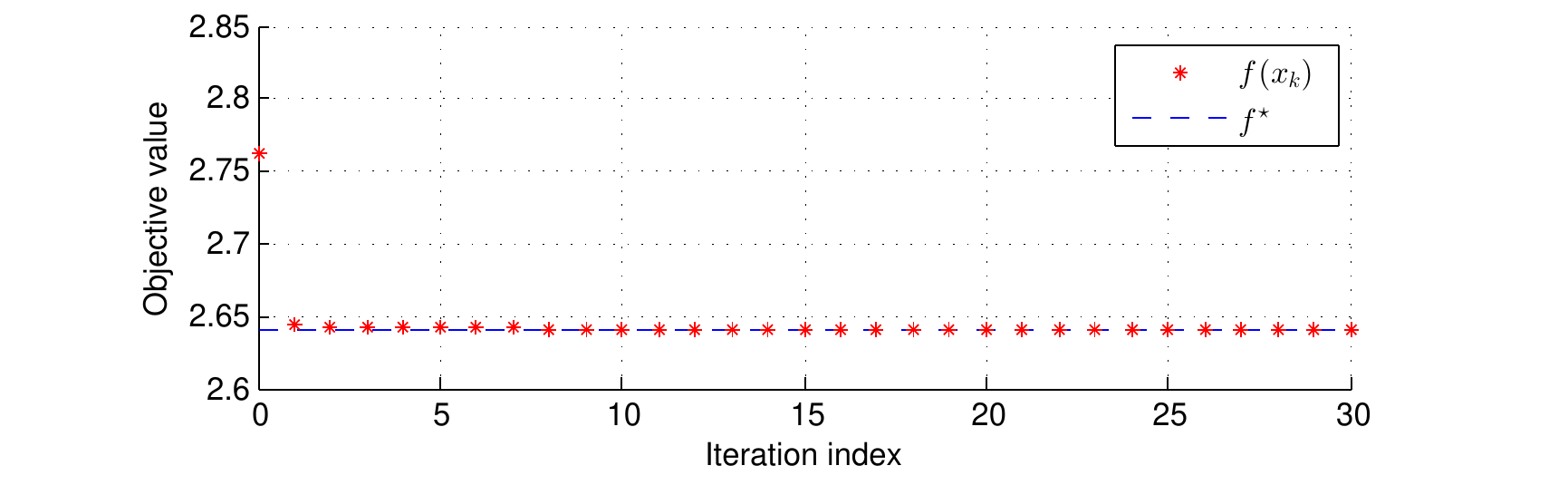} \label{fig:Es1_f1a}}\\
\subfloat[Iteration error $\|x_{k}-x_{k-1}\|$ (blue squares).]{\includegraphics[trim=2cm 14cm 2cm 8.75cm, clip=true,scale=0.55]{./Es1_N100_f1r.pdf}\label{fig:Es1_f1b}}
\caption{Objective value and iteration error for $m=100$, $\gamma^i \in [0.1, 0.3]$, $\underline{x}^{i}(t)=0$ and $\overline{x}^{i}(t)=0.02$, for $i=1,\ldots,m$, $t=0,\ldots,T$.}
\end{figure}

\begin{figure}[t!]
\centering
\includegraphics[trim=2cm 7cm 2cm 7cm, clip=true,scale=0.45]{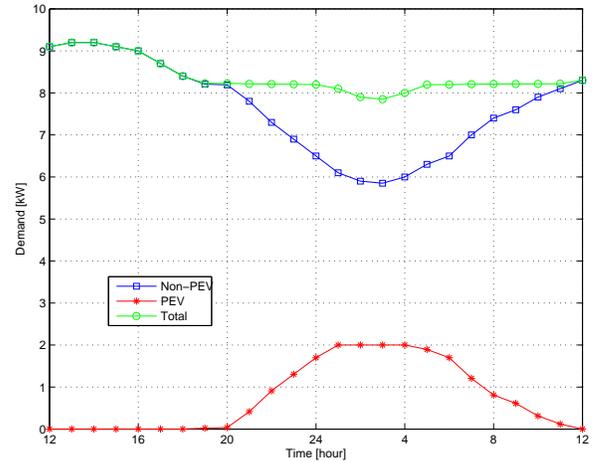}\\
\caption{Demand along a one day time horizon: non-PEV demand (blue squares), PEV demand computed via Algorithm \ref{alg:Alg1} (red stars), and total demand (green circles). The simulation set-up corresponds to $m=100$, $\gamma^i \in [0.1, 0.3]$, $\underline{x}^{i}(t)=0$ and $\overline{x}^{i}(t)=0.02$, for $i=1,\ldots,m$, $t=0,\ldots,T$.}
\label{fig:Es1_f2}
\end{figure}

\begin{table}[t!]
  \centering
  \caption{Number of iterations $k$ for $\frac{f(x_k)-f^\star}{f^\star}<10^{-6}$, for different values of $c$. }
    \begin{tabular}{|l|c|c|c|c|c|c|c|c|}
    \hline
 $c$ &  0 & 0.05 & 0.075 & 0.1 & 0.1478 & 0.2 &0.4 \\
 \hline
$k$  & - & - & 10  & 16 & 27 & 37 & 77 \\
    \hline
    \end{tabular}%
  \label{tab:conv}%
\end{table}%

Figure \ref{fig:Es1_f2} depicts the PEV, non-PEV, and total demand along the considered time horizon. As it can be seen, the PEV demand is optimized so that the over-night valley of the non-PEV demand is nearly filled-up. Note that due to the constraints in \eqref{pb:EVpb2}, it is not possible to further increase the PEV demand during the time interval between 1 and 4.

Considering the same setting, we perform a parametric analysis, running Algorithm \ref{alg:Alg1} for different values of $c$. In Table \ref{tab:conv} the number of iterations needed to achieve a relative error between the decentralized and the centralized objective value $\frac{f(x_k)-f^\star}{f^\star}<10^{-6}$ is reported. It can be observed that, as $c$ increases, numerical convergence requires more iterations.
Note that if we choose a value of $c$ that does not satisfy \eqref{eq:coef_c_quad}, Algorithm \ref{alg:Alg1} does not always converge (see the first two rows of Table \ref{tab:conv}, for $c = 0$ and $c=0.05$). For some values of $c$ numerical convergence is achieved (e.g., rows 3 and 4 of Table \ref{tab:conv}); however, this is not guaranteed by our analysis.

We consider now a fleet of $m = 1000$ PEVs, and modify the required charging level such that $\gamma^i \in [0.005, 0.025]$, for all $i=1,\ldots,m$, and set the bounds on $x^{i}(t)$ to $\underline{x}^{i}(t)=0$ and $\overline{x}^{i}(t)=0.0025$, for all $i=1,\ldots,m$, $t=0,\ldots,T$.
All the other parameters are left unchanged with respect to the previous set-up.
F
After 30 iterations the difference between the decentralized and the centralized objective is $f(x_{30})-f^\star=8.18\cdot10^{-7}$, thus achieving numerical convergence. However, with the new set of parameters, the peak of the PEV demand is designed so as to fill the over-night valley of the non-PEV demand.

It should be noted that our algorithm converges to a minimizer of \eqref{pb:EVpb1}, which is the cooperative problem counterpart of \cite{Grammatico_etal_2015,Franc_2014,Hiskens_2013}, with a finite number of agents/vehicles (see Figure \ref{fig:Es2_f3}), as opposed to the aforementioned references, where convergence to a Nash equilibrium at the limiting case of an infinite population of agents is established. Note also that the algorithm proposed in \cite{Gan_etal_2013} ensures convergence to the optimal value with a finite number of agents, but not to the optimal solution.

\begin{figure}[t!]
\centering
\includegraphics[trim=2cm 8cm 2cm 8.2cm, clip=true,scale=0.5]{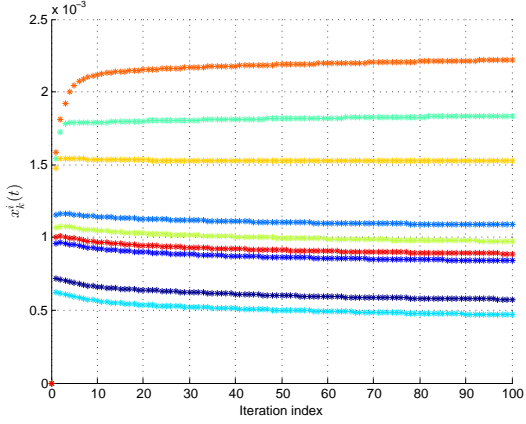}\\
\caption{Evolution of the iterates $x_k^i(t)$ generated by Algorithm \ref{alg:Alg1} at $t=12$ as a function of the iteration index $k$, for $i=1,\ldots,10$, i.e., the first $10$ vehicles of the $1000$-vehicle fleet. }
\label{fig:Es2_f3}
\end{figure}

\section{Concluding remarks} \label{sec:secVI}
In this paper, we investigated convergence of a decentralized, regularized Jacobi algorithm for multi-agent, convex optimization programs with a common objective and subject to separable constraints. It was shown that  in the case where the objective function is quadratic the algorithm converges to some minimizer of the centralized problem via fixed point arguments. In the more general case of a convex but not necessarily quadratic cost, we were able to show that all limit points of the proposed algorithm are minimizers of the centralized problem.
The efficacy of the proposed algorithm was illustrated by applying it to the problem of optimal charging of electric vehicles, achieving convergence with a finite number of vehicles.
Current work concentrates on investigating the rate of convergence for general convex functions, exploiting the recent fixed-point theoretic results of \cite{BMG_2016}.

\appendix
%

\section{Algorithm analysis for differentiable convex objective functions} \label{sec:app2}


We report some identities which hold directly from the definition of $f$ and its gradient, that will be exploited in the following.
\begin{align}
\sum_{i=1}^m f(x^i,x^{-i}) &= m f(x), \label{eq:if1} \\
\nabla \Big ( \sum_{i=1}^m f(\cdot,x^{-i}) \Big ) \Big |_x &= \nabla f(x), \label{eq:if2} \\
 \nabla \Big ( \sum_{i=1}^m f(x^i,\cdot) \Big ) \Big |_x &= (m-1) \nabla f(x), \label{eq:if3}
\end{align}
where by the notation $\nabla\Big ( \sum_{i=1}^m f(\cdot,x^{-i}) \Big ) \Big |_x $ (similarly for $\nabla\Big ( \sum_{i=1}^m f(x^i,\cdot) \Big ) \Big |_x $), we imply that the gradient of $\sum_{i=1}^m f(\cdot,x^{-i})$ is evaluated at $x$, where for each $i$, $i=1,\ldots,m$, the appropriate component $x^i$ of $x$ is employed. Similar is the interpretation of the various gradient terms appearing in the subsequent derivations.

\subsubsection*{Proof of Lemma \ref{lm:Lip_constant}}
Fix any $x,y,z \in X \subset \mathbb{R}^n$, with $x = (x^1,\ldots,x^m)$, $y = (y^1,\ldots,y^m)$, and $z = (z^1,\ldots,z^m)$.
We then have that
\begin{align}
&\|\nabla \Big(\sum_{i=1}^m f(\cdot,x^{-i}) \Big ) \Big|_{z} - \nabla\Big(\sum_{i=1}^m f(\cdot,y^{-i}) \Big ) \Big|_{z}\| \nonumber \\
&\hspace{1cm}  = \| \sum_{i=1}^m\Big( \nabla f(\cdot,x^{-i}) - \nabla  f(\cdot,y^{-i}) \Big ) \Big|_{z}\| \nonumber \\
&\hspace{1cm} = \| \sum_{i=1}^m  \begin{bmatrix} 0\\\vdots\\ \Big(\nabla^i f(\cdot,x^{-i}) - \nabla^i  f(\cdot,y^{-i}) \Big )\Big|_{z}\\ \vdots\\ 0\end{bmatrix} \| \nonumber \\
&\hspace{1cm}=\|\begin{bmatrix} \nabla^1 f(z^1,x^{-1}) - \nabla^1  f(z^1,y^{-1}) \\\vdots\\  \nabla^i f(z^i,x^{-i}) - \nabla^i  f(z^i,y^{-i}) \\ \vdots\\ \nabla^m f(z^m,x^{-m}) - \nabla^m  f(z^m,y^{-m})\end{bmatrix}\| , \label{eq:Lip_constant_pf1}
\end{align}
where the first equality is obtained by exchanging the gradient and the summation order, since the gradient is a linear operator, the second equality is due to the fact that, for each $i$, $i=1,\ldots,m$, all components of $\nabla f(\cdot,x^{-i})$ (similarly for $\nabla f(\cdot,y^{-i})$ ) will be zero apart from the $i$-th one that will be $\nabla^i f(\cdot,x^{-i})$, and the last equality is achieved by performing the summation and evaluating the quantity in the parentheses at $z$. We then have that
\begin{align}
&\sqrt{ \sum_{i=1}^m \| \nabla^i f(z^i,x^{-i}) - \nabla^i  f(z^i,y^{-i}) \|^2 } \nonumber \\
&\leq \sqrt{\sum_{i=1}^m L^2 \|x^{-i} - y^{-i} \|^2 } \leq \sqrt{m} L \|x-y\| 
\end{align}
where the first expression follows from \eqref{eq:Lip_constant_pf1}, the first inequality is due to fact that, as a consequence of Assumption \ref{ass:grad_Lip}, for each $i=1,\ldots,m$, the $i$-th component of the gradient is Lipschitz continuous with a Lipschitz constant upper-bounded by $L$, i.e., $\|\nabla^if(x)-\nabla^if(y)\| \le \|\nabla f(x)-\nabla f(y)\| \le L \|x-y\|$, for all $x, y \in X$. The last inequality follows by noticing that $\|x^{-i} - y^{-i}\| \leq \|x-y\|$. Combining \eqref{eq:Lip_constant_pf1} and the inequality above 
leads then to \eqref{eq:Lip_constant}, and hence concludes the proof.

\subsubsection*{Proof of Theorem \ref{thm:Alg_conv}}
From step 6 of Algorithm \ref{alg:Alg1} and by the definition of $\widetilde{T}$ in \eqref{eq:Tmapsum}, we have that $x_{k+1} = \widetilde{T}(x_{k})$. Therefore,
$x_{k+1}$ is optimal for the objective function that appears in the right-hand side of \eqref{eq:Tmapsum}, hence we have that, for all $y \in X$,
\begin{align}
\sum_{i=1}^m f(x_{k+1}^i,x_k^{-i}) &+ c\|x_{k+1}-x_{k}\|^2 \nonumber \\ &\le \sum_{i=1}^m f(y^i,x_k^{-i})+c\|y^i-x_{k}\|^2.
\end{align}
Since $x_k \in X$, setting $y=x_{k}$ and rearranging some terms, we have that
\begin{align}\label{eq:OP2}
\sum_{i=1}^m f(x_{k+1}^i,x_k^{-i}) \le \sum_{i=1}^m f(x_k^{i},x_k^{-i}) -c \|x_{k+1}-x_k\|^2.
\end{align}
By convexity of $\sum_{i=1}^m f(x_{k+1}^i,\cdot)$ due to Assumption \ref{ass:grad_Lip}, we have that
\begin{align}\label{eq:cvx-2}
\sum_{i=1}^m f(x_{k+1}^i,&x_k^{-i}) \ge \sum_{i=1}^m f(x_{k+1}^i,x_{k+1}^{-i}) \nonumber \\ &+ (x_{k}-x_{k+1})^\top \nabla\Big(\sum_{i=1}^m f(x_{k+1}^{i},\cdot) \Big ) \Big|_{x_{k+1}}.
\end{align}
Combining \eqref{eq:OP2} and \eqref{eq:cvx-2}, and after rearranging some terms, we obtain that
\begin{align}
\sum_{i=1}^m f(&x_{k+1}^i,x_{k+1}^{-i}) \le \sum_{i=1}^m f(x_k^{i},x_k^{-i}) -c\|x_{k+1}-x_k\|^2 \nonumber \\
&+ (x_{k+1}-x_{k})^\top \nabla\Big(\sum_{i=1}^m f(x_{k+1}^{i},\cdot) \Big ) \Big|_{x_{k+1}},
\end{align}
which in view of \eqref{eq:if1} and \eqref{eq:if3} can be rewritten as
\begin{align}\label{eq:DL4-2}
m f(x_{k+1}) &\le m f(x_k) -c\|x_{k+1}-x_k\|^2 \nonumber \\ &+ (m-1)(x_{k+1}-x_k)^\top \nabla f(x_{k+1}).
\end{align}

Since $x_{k+1} = \widetilde{T}(x_k)$ and due to \eqref{eq:Tmapsum}, by the first part of Proposition \ref{prop:opt_VI}, with the objective function at the right-hand side of \eqref{eq:Tmapsum} in place of $f$, we obtain that
\begin{align}
(x_k - x_{k+1})^\top \nabla \Big(\sum_{i=1}^m f(\cdot,&x_k^{-i}) \Big ) \Big|_{x_{k+1}} \nonumber \\ &- 2c \|x_{k+1}-x_k\|^2 \geq 0. \label{eq:VI1}
\end{align}
Note that $\nabla \Big(\sum_{i=1}^m f(\cdot,x_k^{-i}) \Big ) \Big|_{x_{k+1}} + 2c (x_{k+1} - x_k)$ is the gradient of the objective function at the right-hand side of \eqref{eq:Tmapsum}, evaluated at $x_{k+1}$.
Adding the term $(x_{k+1}-x_k)^\top \nabla\Big(\sum_{i=1}^m f(\cdot,x_{k+1}^{-i}) \Big ) \Big|_{x_{k+1}}$ in both sides of \eqref{eq:VI1}, we obtain that
\begin{align*}
&(x_{k+1}-x_k)^\top \nabla\Big(\sum_{i=1}^m f(\cdot,x_{k+1}^{-i}) \Big ) \Big|_{x_{k+1}} \le -2c\|x_{k+1}-x_{k}\|^2 \nonumber \\
&- (x_{k+1}-x_k)^\top \Bigg [ \nabla\Big(\sum_{i=1}^m f(\cdot,x_{k}^{-i}) \Big ) \Big|_{x_{k+1}}  \nonumber \\ & \hspace{4cm} - \nabla\Big(\sum_{i=1}^m f(\cdot,x_{k+1}^{-i}) \Big ) \Big|_{x_{k+1}}  \Bigg ].
\end{align*}
By the Cauchy-Schwarz inequality, and due to Lemma \ref{lm:Lip_constant}, which holds under Assumption \ref{ass:grad_Lip}, we have that
\begin{align}
(x_{k+1}-x_k)^\top \nabla\Big(\sum_{i=1}^m &f(\cdot,x_{k+1}^{-i}) \Big ) \Big|_{x_{k+1}} \nonumber \\ &\le (\sqrt{m}L-2c)\|x_{k+1}-x_{k}\|^2,
\end{align}
which in view of \eqref{eq:if2} can be rewritten as
\begin{align}\label{eq:VI1e-2}
(x_{k+1}-x_k)^\top \nabla &f(x_{k+1}) \le (\sqrt{m}L-2c)\|x_{k+1}-x_{k}\|^2.
\end{align}
By \eqref{eq:DL4-2} and \eqref{eq:VI1e-2}, we then have that
\begin{align}\label{eq:DL5-2}
mf(&x_{k+1}) \le m f(x_k) \nonumber \\ & + \big ( -c +(m-1)(\sqrt{m}L-2c) \big ) \|x_{k+1}-x_k\|^2.
\end{align}
If we choose the regularization coefficient $c$ according to \eqref{eq:coef_c},
then the term $\|x_{k+1}-x_k\|^2$ in \eqref{eq:DL5-2} is multiplied by a negative constant, and hence \eqref{eq:DL5-2} shows that $f(x_k)$ is a non-increasing sequence, i.e., for all $k \geq 0$, $f(x_{k+1}) \leq f(x_k)$.

Moreover, by \eqref{eq:DL5-2} and for a given $N \in \mathbb{R}_+$, we have that
\begin{align}
& \sum_{k=1}^N -\big ( -c +(m-1)(\sqrt{m}L-2c) \big ) \|x_{k+1}-x_k\|^2 \nonumber \\
&\leq m \sum_{k=1}^N \big ( f(x_k) - f(x_{k+1}) \big ) = m \big (f(x_1) - f(x_{N+1}) \big ).
\end{align}
Notice that $f(x_1) - f(x_{N+1}) \geq 0$ due to the fact that $\{f(x_k)\}_{k\geq 0}$ is non-increasing.
If $c$ is chosen then according to \eqref{eq:coef_c}, and letting $N \to \infty$, the last statement implies that $\sum_{k=1}^\infty \|x_{k+1}-x_k\|^2 < \infty$, thus $\lim_{k \to \infty} \| x_{k+1} -x_k \| = 0$, and hence
\begin{align}
\lim_{k \to \infty} \| \widetilde{T}(x_k) -x_k \| = 0. \label{eq:iterate_conv}
\end{align}

However, by \eqref{eq:DL5-2}, and under the choice of $c$ in \eqref{eq:coef_c}, we have that $\{f(x_k)\}_{k \geq 0}$ is a monotonically decreasing sequence, i.e., $f(x_{k+1}) < f(x_k)$ for all $k \geq 0$, but for the case where $x_{k+1} = x_k$ for some $k \geq 0$.
The latter case implies that there exists $\bar{k} \geq 0$ such that $x_{\bar{k}+1} = x_{\bar{k}}$. Since $x_{\bar{k}+1} = \widetilde{T}(x_{\bar{k}})$, the last statement implies that $x_{\bar{k}}$ is a fixed-point of $\widetilde{T}$, and, due to Corollary \ref{cor:M_FP_Ttilde}, will be a minimizer of $\mathcal{P}$.

Consider now the former case, where $\{f(x_k)\}_{k \geq 0}$ is a monotonically decreasing sequence. Since $x_k \in X$ for all $k \geq 0$, and $X$ is compact due to Assumption \ref{ass:grad_Lip}, $\{f(x_k)\}_{k \geq 0}$ is bounded below. Therefore, the sequence $\{f(x_k)\}_{k \geq 0}$ is convergent. Consider any convergent subsequence of $\{x_k\}_{k \geq 0}$ and denote by $x^*$ its limit point. By \eqref{eq:iterate_conv}, we then have that $\widetilde{T}(x_k)$ converges to $x^*$ as well. The latter, together with the fact that $\widetilde{T}$ is continuous implies that $\widetilde{T}(x^*) = x^*$, i.e., $x^*$ is a fixed-point of $\widetilde{T}$.
Note that since $\widetilde{T}$ is single-valued as an effect of being strictly convex due to the presence of the regularization term, continuity follows from 5.22 in \cite{Rockafellar_Wets_1998} (p. 162) if the objective function in \eqref{eq:Tmapsum} is continuous differentiable, the set $X$ is compact, and the minimum in \eqref{eq:Tmapsum} as a function of $x$ is continuous. Continuous differentiability and compactness emanate from Assumption \ref{ass:grad_Lip}, whereas it is shown in Exercise 1.19 in \cite{Rockafellar_Wets_1998} (p. 18) that, under the stated assumptions and since the constraint set does not depend on $x$, the minimum is a continuous function of $x$. Hence,  $\widetilde{T}$ is continuous and as shown above $x^*$ belongs to its fixed-points.
Due to the equivalence between the fixed-points of $\widetilde{T}$ and the minimizers of $\mathcal{P}$ shown in Corollary \ref{cor:M_FP_Ttilde}, the limit point $x^*$ of $\{x_k\}_{k \geq 0}$ is also a minimizer of $\mathcal{P}$, hence $\lim_{k \to \infty} \mathrm{dist}(x_k,X^*) = 0$, thus concluding the proof.

It should be noted that the last part of the proof is motivated by the arguments in p. 214, Chapter 3 of \cite{Bertsekas_Tsitsiklis_1997}.

\subsubsection*{Proof of Lemma \ref{lm:Lip_constant_quad}}

Under Assumption \ref{ass:Convexity}, and by the definition of $Q_d, Q_z$, we have that for any $x,z \in X$,
$\sum_{i=1}^m f(z^i,x^{-i})
= \sum_{i=1}^m (z^i)^\top Q_{i,i} z^i + 2 (x^{-i})^\top Q_{-i,i} z^i + q_i^\top z^i 
 = z^\top Q_d z + 2 x^\top Q_z z + q^\top z$,
where in the last step we used the fact that $Q_z = Q_z^\top$. By the last statement we then have that
$\nabla \Big(\sum_{i=1}^m f(\cdot,x^{-i}) \Big ) \Big|_{z} = 2 Q_d z + 2 Q_z x + q.$
Following an analogous derivation with $y \in X$ in place of $x$, we have that
$\nabla \Big(\sum_{i=1}^m f(\cdot,y^{-i}) \Big ) \Big|_{z} = 2 Q_d z + 2 Q_z y + q$.
Therefore,
\begin{align}
\|\nabla \Big(&\sum_{i=1}^m f(\cdot,x^{-i}) \Big ) \Big|_{z} - \nabla\Big(\sum_{i=1}^m f(\cdot,y^{-i}) \Big ) \Big|_{z}\| \nonumber \\ &= \| 2 Q_z (x-y) \|  \leq 2 \lambda^{\max}_{Q_z} \|x-y\|,
\end{align}
thus concluding the proof.

\end{document}